\documentclass[12pt,a4paper]{iopart}

\usepackage{epsfig}
\usepackage{iopams,amsthm}

\newcommand{\tp}{\theta,\phi}
\newcommand{\R}{\mathbf R}
\newcommand{\C}{{\mathbb C}}
\newcommand{\Z}{{\mathbb Z}}
\newcommand{\F}{{\mathcal F}}
\renewcommand{\O}{{\mathcal O}}
\renewcommand{\O}{{\mathcal O}}

\renewcommand{\Im}{{\text{Im}}}
\renewcommand{\Re}{{\text{Re}}}

\newcommand{\nn}{\nonumber}
\newcommand{\Texp}{\mathbf{t}^{\mbox{{\tiny exp}}}}

\renewcommand{\t}{{\mathbf{t}}}
\newcommand{\txz}{{\mathbf{t}}(\xi,\zeta)}
\newcommand{\qexp}{q^{\mbox{{\tiny exp}}}}
\newcommand{\qhat}{\hat q}
\newcommand{\gammaexp}{\mathbf{\gamma}^{\mbox{{\tiny exp}}}}

\newcommand{\uexp}{{u}^{\mbox{{\tiny exp}}}}
\newcommand{\du}{\delta u}
\newcommand{\gammaapp}{\mathbf{\gamma}^{\mbox{{\tiny app}}}}

\newcommand{\qp}{{q}^{\mbox{{\tiny p}}}}
\newcommand{\mup}{{\mu}^{\mbox{{\tiny p}}}}
\newcommand{\gp}{{g}^{\mbox{{\tiny p}}}}

\newcommand\eqref[1]{(\ref{#1})} 
\newcommand{\text}{\textnormal }

\newcommand{\Om}{\Omega}
\newcommand{\DOm}{\partial\Omega}
\newcommand{\ol}{\overline}
\newcommand{\V}{{\mathcal{V}}} 
\newcommand{\RRR}{\mathbf{R}}

\newcommand{\supp}{\text{supp}}

\newtheorem{thm}{Theorem}[section]
\newtheorem{cor}[thm]{Corollary}

\newtheorem{prop}[thm]{Proposition}

\begin{document}

\title[Direct numerical reconstruction of conductivities in 3-D]{Direct numerical reconstruction of
  conductivities in three dimensions}

\author{Jutta Bikowski$^1$, Kim Knudsen$^2$, Jennifer Mueller$^3$}

\address{1\ Agrosphere, ICG-4, Forschungszentrum J\"ulich, J\"ulich, Germany }

\address{2\ Department of Mathematics, Technical University
  of Denmark, Kgs.\ Lyngby, Denmark}

\address{3\ 
 Department of Mathematics, Colorado State University, Fort Collins} 
\ead{jubi305@gmail.com, k.knudsen@mat.dtu.dk, mueller@math.colostate.edu}


\begin{abstract}
  A direct three dimensional EIT reconstruction algorithm based on complex geometrical
  optics solutions and a nonlinear scattering transform is presented
  and implemented for spherically symmetric conductivity
  distributions.  The scattering transform is computed both with a
  Born approximation and from the forward problem for purposes of
  comparison.  Reconstructions are computed for several test problems.
  A connection to Calder\'on's linear reconstruction algorithm is established, and
  reconstructions using both methods are compared.
\end{abstract}

\maketitle

%Version 20, March 19, 2010 submitted

\section{Introduction}

The reconstruction of conductivity distributions in two or three
dimensions from measurements of the current density-to-voltage map is known as
electrical impedance tomography, or EIT, and has applications in
medical imaging, nondestructive testing, and geophysics. For the 3-D
bounded domain considered here, medical applications include head
imaging and the detection of breast tumors.  See, for example,
\cite{holder_book} for a survey of clinical applications of EIT. In
this work, we consider a bounded domain in $\R^3$ and present a direct reconstruction algorithm and its
numerical implementation on the unit sphere. The theoretical foundation of the method dates back
more than 20 years to a series of papers by
Sylvester-Uhlmann \cite{sylvesterUhlmann87}, Novikov \cite{novikov},
Nachman-Sylvester-Uhlmann \cite{NSU} and Nachman \cite{nachman88}.
The algorithm makes use of complex geometrical optics (CGO) solutions
to the Schr\"odinger equation and uses the inverse scattering method.  This is described in detail
in section 2 of this paper.

The inverse conductivity problem was first formulated mathematically
by Calder\'on \cite{Calderon} as follows. Let $\Omega \subset \R^n, n
\geq 3$ be a simply connected, bounded domain with smooth boundary
$\partial \Om,$ and let $\gamma \in L^\infty(\Omega)$ denote the
conductivity distribution. Assume there exists $C>0$ such that for $x\in\Omega$, 
  $C^{-1} \leq \gamma(x) \leq C$.  
The electric potential $u$ arising from the application of a known voltage to the boundary of $\Omega$ is  
modeled by the generalized Laplace equation with Dirichlet boundary condition
\begin{eqnarray}  \label{eq:cond}
  \nabla \cdot \gamma \nabla u = 0 \mbox{ in } \Omega,\qquad u = f  \mbox{ on } \partial \Omega.
\end{eqnarray}
The Dirichlet-to-Neumann map $\Lambda_\gamma$
is defined by
\begin{eqnarray}  \label{DNmap}
  \Lambda_\gamma f = \gamma \frac{\partial u}{\partial \nu}\vert_{\partial \Omega}.
\end{eqnarray}
Thus $\Lambda_\gamma$ represents static electrical boundary
measurements: it maps an applied voltage distribution on the boundary
to the resulting current flux through the boundary. Calder\'on
\cite{Calderon} posed the question of whether the conductivity
$\gamma$ is uniquely determined by the Dirichlet-to-Neumann map, and
if so, how to reconstruct the conductivity. He gave an affirmative
answer to the uniqueness question for the linearized problem and gave
a reconstruction algorithm for that case.  His algorithm is described
in section~\ref{sec:calderon} of this paper.

The uniqueness question for $\gamma\in L^{\infty}(\Om)$ is still open
in $\R^3$, but has been solved recently by Astala and P\"aiv\"arinta
\cite{astalaPaivarinta} for a bounded domain in $\R^2$, sharpening the previous results due to Nachman 
\cite{nachman96} in which $\gamma\in W^{2,p}(\Om)$, $p>1,$ and Brown
and Uhlmann \cite{brownUhlmann} in which $\gamma\in W^{1,p}(\Om)$, $p>2$. 
%The reconstruction problem was solved in theory in \cite{astalaPaivarinta2} improving the results in \cite{nachman96} and \cite{knudsenTamasan}. 
In three dimensions the uniqueness problem was solved for smooth conductivities
in \cite{sylvesterUhlmann87}. At the time of
this publication, in $\R^3$ the uniqueness results with lowest
regularity are \cite{brownTorres} with $\gamma\in W^{3/2,p}(\Om)$, $p>2n$ 
and \cite{PPU} with $\gamma\in W^{3/2,\infty}(\Om)$.

Most existing 3-D EIT reconstruction algorithms are linear or
iterative, minimizing a functional that describes the
nearness of the predicted voltages to the measured data in a given
norm with one or more regularization terms.  In contrast, the algorithm presented here is direct and fully nonlinear.
It is similar to the 2-D D-bar algorthims based on the works \cite{nachman96} and \cite{brownUhlmann}, which were first
implemented in \cite{SMI00, ams, knudsen03,MS03}.  In these initial works the Born approximation to the CGO solutions is used in the 
computation of the scattering transform.  It was used successfully on experimental tank data in, for
example, \cite{isaacson_reconstruct_chest_2004,murphy_2009} and human chest data in \cite{isaacson_imaging_cardiac_2006,DeAngeloMueller}.  This
inspired the approach in section \ref{sec:texp} of this paper in which the Born approximation is used in the 3-D direct algorithm.  For further
reading on 2-D D-bar algorithms, the reader is referred to 
\cite{knudsenLassasMuellerSiltanen07} in which the application to discontinuous conductivity distributions is specifically addressed, 
and \cite{knudsenLassasMuellerSiltanen09} in which a rigorous regularization
framework is established using the full scattering transform.
Calder\'on's method has
also recently been used for the reconstruction from experimental data
in both 2-D \cite{bikowskiMueller} and 3-D  \cite{Greg_Dave}.

%While iterative methods have the potential to converge to local minima, direct methods do not
%suffer from this shortcoming.  Thus, they have the potential to yield
% more accurate conductivity values and better spatial resolution. %  In
% this paper, we derive the explicit connection between Calder\'on's
% algorithm and the method presented here.  Numerical comparisons on the
% spherically symmetric examples are provided.  In a recent paper
% \cite{Greg_Dave} Calder\'on's method is applied in order to
% reconstruct the conductivity between two parallel plates of electrodes
% for breast cancer detection.

In this work, we assume the conductivity $\gamma\in C^2(\ol \Omega),$ we
take $\Omega$ to be the unit sphere in $\R^3,$ and assume $\gamma = 1$
near $\partial\Omega.$ The smoothness assumption on $\gamma$ is necessary, but
the other assumptions are made mainly for simplicity in the numerical
computations. We stress in particular that the theory is valid in more
complex geometries.  The study of the effects of noise in the data is not in the scope of this
paper, but rather is left for future work.

The outline of the paper is as follows. In section 2.2 we describe a
direct reconstruction algorithm with a linearizing assumption
tantamount to a Born approximation. That approach is referred to as
the $\Texp$ approach, consistent with the notation used in the 2-D
D-bar algorithms.  An explicit connection to the linearized method of
Calder\'on is established in section \ref{sec:calderon}.  The
reconstruction of the conductivity in the 2-D D-bar method described
in the works above is achieved by taking a small frequency limit in a
D-bar equation for the CGO solutions to directly obtain $\gamma(x)$.
In contrast, here we have to take a high complex frequency limit.  A
D-bar equation for the 3-D problem is utilized in
\cite{corneanKnudsenSiltanen}, resulting in a promising, but more
complicated approach than the one studied here.  The numerical
implementation of that approach is left for future work.  In section
\ref{sec:spherical} we consider the case of spherically symmetric
conductivities and show symmetry properties in the scattering
transform. We also show how the Dirichlet-to-Neumann map can be
represented and approximated in that case. Details on the numerical
implementation are found in section \ref{sec:numerics}.  Numerical
examples are found in section \ref{sec:examples}.

\section{The reconstruction methods} \label{sec:recon} In this section
we describe the theoretical reconstruction method, the $\Texp$ approach, and
the relationship to Calder\'on's linearized method.

\subsection{The nonlinear reconstruction method}\label{sec:nonlinear}
The method was developed in
\cite{sylvesterUhlmann87,novikov,NSU,nachman88}; here we provide a
brief outline.  The reader is referred to \cite{nachman88} for
rigorous proofs.  The equations closely parallel those of the 2-D
problem (note that \cite{nachman88} precedes that work), and so
readers familiar with that case will recognize the notation and
functions involved.

The initial step is to transform the conductivity equation into a
Schr\"odinger equation. Indeed, if $u$ satisfies \eqref{eq:cond} then
$v= \gamma^{1/2}u$ satisfies
\begin{eqnarray}\label{eq:schr}
  (-\Delta + q) v = 0 \text{ in } \Omega ~~~ \mbox{with} ~~~  q = \frac {\Delta \gamma^{1/2}}{\gamma^{1/2}}.
\end{eqnarray}
Note that $q = 0$ near $\partial\Omega.$ The Dirichlet-to-Neumann map
for equation \eqref{eq:schr} is defined by
 $$ \Lambda_q f =  \frac{\partial v}{\partial \nu}|_{\partial \Omega},$$
where now $v$ satisfies \eqref{eq:schr} with $v|_{\partial\Omega} = f.$
In general the maps $\Lambda_\gamma$ and $\Lambda_q$ are  related by
\begin{eqnarray} \label{DtNmaps}
\Lambda_q = \gamma^{-1/2}\left(\Lambda_{\gamma}+\frac{1}{2}\frac{\partial\gamma}{\partial\nu}\right)\gamma^{-1/2}.
\end{eqnarray}
The assumption that $\gamma = 1$ in a neighborhood of $\partial \Omega$ simplifies \eqref{DtNmaps} 
to $\Lambda_q = \Lambda_\gamma.$

To define the CGO solutions, introduce a complex frequency parameter
$\zeta \in \C^3$ and define the set
\begin{eqnarray} \label{vdef}
\V = \{\C^3\setminus\{0\} \colon \zeta\cdot\zeta =0\}.
\end{eqnarray}
Then $e^{ix\cdot\zeta}$ is
 harmonic in $\R^3$ if and only if
$\zeta \in \V.$ For $\xi\in\R^3$, introduce the subset of $\V$ given by
\begin{eqnarray}
  \label{eq:Vxi}
   \V_\xi = \{ \zeta \in \V \colon (\xi + \zeta)^2 = 0\}.
\end{eqnarray}
Note that $\zeta\cdot \zeta = (\xi + \zeta)^2 = 0$ gives an explicit
characterization of $\V_\xi$ in terms of an auxiliary vector
$\xi^\perp \in \R^3$ with $ \xi^\perp\cdot \xi = 0.$ Indeed suppose
$\zeta_R, \zeta_I \in \R^3.$ Then $\zeta = \zeta_R + i
\zeta_I\in\V_\xi$ if and only if
\begin{eqnarray}\label{zetaform}
     \zeta_R &= -\xi/2 +
    \xi^\perp,\\
    \zeta_I \cdot \xi = \zeta_I \cdot \xi^\perp &= 0 ,\quad |\zeta_I| = |\zeta_R|.\nn
\end{eqnarray}

Since $q=0$ in a neighborhood of $\DOm$, one can extend $q=0$ into $\R^3\setminus \ol\Omega.$ The CGO solutions
$\psi(x,\zeta)$ to the Schr\"odinger equation solve
\begin{eqnarray}
  \label{schrodinger}
  (-\Delta + q(x))\psi(x,\zeta) = 0, \quad x\in \R^3, \quad \zeta \in \V,
\end{eqnarray}
and behave like $e^{ix\cdot\zeta}$ for $|\zeta|$ large. More precisely,
define 
$$\mu(x,\zeta) = \psi(x,\zeta) e^{-ix\cdot\zeta}.$$
 Then $\mu-1$ approaches
zero in  a certain sense as either $|x|$ or $|\zeta|$ tends to infinity, see \cite{sylvesterUhlmann87,nachman88}.  Note that
$\psi(x,\zeta)$ grows exponentially for $x\cdot\Im \zeta <0$. The function $\mu$ satisfies
\begin{eqnarray}\label{mueq}
  (-\Delta -2i\zeta\cdot\nabla +  q) \mu(x,\zeta) = 0 \text{ in } \R^3.
\end{eqnarray}

Denote by $G_\zeta$ the Faddeev Green's function defined by
\begin{eqnarray}\label{Gzeta}
  G_\zeta (x) = e^{ix\cdot \zeta}g_\zeta(x), ~~~ \mbox{where} ~~~ g_\zeta(x) = \frac{1}{(2\pi)^3}\int_{\R^3} \frac{e^{ix\cdot
  \xi}}{|\xi|^2 + 2\xi\cdot \zeta}d\xi,
\end{eqnarray}
 where the integral is understood in the sense of the Fourier transform defined on the space of tempered
distributions. The functions $G_\zeta,g_\zeta$ are fundamental
solutions of the Laplace equation
and conjugate Laplace equation respectively, i.e.\
\begin{eqnarray}
  \Delta_\zeta g_\zeta = -\delta_0 \quad \mbox{and} \quad  \Delta G_\zeta = -\delta_0.
\end{eqnarray}
Then \eqref{mueq} is equivalent to the Faddeev-Lippmann-Schwinger equation
\begin{eqnarray}\label{muLS}
  (I + g_\zeta \ast (q \;\cdot\;))\mu = 1 \text{ in } \R^3.
\end{eqnarray}
Estimates for the operator $g_\zeta \ast$ for large $\zeta$
(\cite{sylvesterUhlmann87}) and small $\zeta$
(\cite{corneanKnudsenSiltanen}) give the existence and uniqueness of
$\mu$ (and therefore $\psi$) for any sufficiently large or small $\zeta \in
\V.$

The key intermediate object in the reconstruction method is the
so-called non-physical scattering transform of the potential $q$
defined for $\xi\in\R^3$ and sufficiently large or small $\zeta \in
\V$ by
\begin{eqnarray}
  \t(\xi,\zeta) = \int_{\Omega}
  e^{-ix\cdot(\xi+\zeta)}\psi(x,\zeta)q(x)dx. \label{t}
\end{eqnarray}
Integrating by parts and assuming that $\zeta \in \V_ \xi$
we find that
\begin{eqnarray}\label{tboundary}
  \t(\xi,\zeta) =
  \int_{\partial\Omega}e^{-ix\cdot(\xi+\zeta)}(\Lambda_q-\Lambda_0)\psi(x,\zeta)d\sigma(x).
\end{eqnarray}
Thus we require $\psi|_{\partial\Omega}$ in order to compute the
scattering transform from the Dirichlet-to-Neumann map. It turns out
that 
$\psi|_{\partial\Omega}$ satisfies a uniquely solvable Fredholm
integral equation of the second kind  on $\partial\Omega$ \cite{novikov,nachman88},  namely, 
\begin{eqnarray}
  \label{eq:bie}
  \psi(x,\zeta)+ \int_{\DOm} G_{\zeta}(x-\tilde{x})
  (\Lambda_q-\Lambda_0)\psi(\tilde{x},\zeta) d\sigma(\tilde{x})= e^{ix\cdot
    \zeta}, \quad x\in\DOm.
\end{eqnarray}

Note from \eqref{t} and the fact that $\psi\sim e^{i\zeta\cdot x}$ that from the scattering transform one can compute the Fourier transform $\hat q$ of
the potential by taking the large frequency limit
\begin{eqnarray}
  \label{eq:zetalim}
  \lim_{|\zeta| \rightarrow \infty}\t(\xi,\zeta) = \hat q (\xi).
\end{eqnarray}

Summary of the reconstruction method:
\begin{enumerate}
\item\label{nonlinearStep1} Solve the boundary integral equation  \eqref{eq:bie} for
  $\psi|_{\partial\Om}$. 
  \item  Compute $\t(\xi,\zeta)$ for $\xi \in \R^3,  \; \zeta \in \V_\xi$ by \eqref{tboundary}.
\item Compute $\hat q (\xi)$ from \eqref{eq:zetalim} .
\item Compute $q$ by inverting the Fourier transform.
\item Compute $\gamma$ by solving $-\Delta\sqrt{\gamma}+q\sqrt{\gamma}=0$ in $\Om,$ $\sqrt{\gamma}|_{\partial\Omega}=1.$ 
\end{enumerate}

We stress that the ill-posedness of the inverse problem is in this
algorithm isolated in the first step.

%\subsection{Determining the Schr\"odinger Potential}
%In order to be able to reconstruct the conductivity from the scattering data we determine first the Schr\"odinger potential and then calculating $\gamma$. The last step is solving $\Delta \sqrt \gamma = q\sqrt \gamma$ with the corresponding boundary condition. To get the Schr\"odinger potential we use  the estimate
%%In this step the scattering data $t(\xi,\zeta)$ is used to find the Schr\"odinger potential $q(x)$.
%%The heart of this step is given by the bound
% \begin{equation}|t(\xi,\zeta) - \hat q(\xi)|\leq \frac{c(\delta,a)}{|\zeta|} \|q\|_\delta^2
%\label{boundtq}
%\end{equation}
%where $\hat \cdot $ indicates the Fourier transform and the weighted norm is given by $\| q\|_\delta^2 \equiv  \int(1+x^2)^\delta |q(x)|^2 dx$. The bound was explicitly given in \cite{Nac88} but the major work was already done in \cite{SU87}. 
%In the limit as $|\zeta|$ gets large we have
%\[ \lim_{|\zeta| \to \infty} t(\xi, \zeta) =  \hat q(\xi). \] 
%Once we have $\hat q(\xi)$ we can take the inverse Fourier transform of it to get $q(x)$. 
%

\subsection{The reconstruction method using $\Texp$}\label{sec:texp}

Inspired by the $\Texp$ approximation in the 2-D D-bar method, an analogous approach can be taken in 3-D.
Approximating $\psi(x,\zeta)$ on the boundary by its asymptotic
behavior  $e^{ix\cdot \zeta}$ eliminates the need for the ill-posed
first step. We define for $\xi\in\R^3, \zeta\in \V_\xi$ 
\begin{eqnarray}\label{Texp}
\Texp(\xi,\zeta)  =  \int_{\partial\Omega}e^{-ix\cdot(\xi+\zeta)}(\Lambda_q-\Lambda_0) e^{ix\cdot \zeta}d\sigma(x).
\end{eqnarray}
This approximation is tantamount to a linearization of the first step in the reconstruction
algorithm above around $\gamma=1$.  Using $\Texp$ for $\t$ in \eqref{eq:zetalim} gives the following simple reconstruction
algorithm:
\begin{enumerate}
\item Compute $\Texp(\xi,\zeta)$ for $\xi \in \R^3, \; \zeta \in
  \V_\xi$ by \eqref{Texp}.
\item Compute  
  \begin{eqnarray}\label{texpzetalim}
    \widehat \qexp (\xi) = \lim_{|\zeta|\rightarrow
      \infty}\Texp(\xi,\zeta)
  \end{eqnarray}
  and then $\qexp$ by inverting the Fourier transform.
\item Compute $\gammaexp$ by solving $-\Delta\sqrt{\gammaexp} +\qexp
  \sqrt{\gammaexp}=0$ in $\Om$  $\sqrt{\gammaexp}|_{\partial\Omega}=1.$
\end{enumerate}
It is not guarenteed from the theory that the limit in \eqref{texpzetalim} is
well-defined. In our numerical simulations we will compute
$\Texp(\xi,\zeta)$ for a fixed but large value of $\zeta.$ This will
numerically define $\widehat \qexp (\xi).$ 
%There are challenges in this algorithm.  First, The function $\Texp$
%seems to be growing exponentially for large $|\xi|$ or large
%$|\zeta|.$ Hence, $\Texp$ may not converge for $|\zeta| \rightarrow
%\infty,$ and even if it does the limit may not be in a reasonable
%function space (say $L^2(\R^3)$). Thus additional regularization has
%to be introduced for this method to be practical. We will in section
%\ref{sec:examples} below pick an arbitrary $\zeta \in \V_\xi$ and truncate $\Texp$
%for large $\xi.$ This will serve the purpose.
% H The approximation $texp$ is good for $|\zeta|$ very large
%because of the asymptotic behaviour of $\psi$. But taking $|\zeta|$
%large requires the work with highly oscillatory function as integrands
%in step 1 which is known to be a numerical challenge. To bypass this
%difficulty one could take smaller values of $|\zeta|$ which might
%result in a difference between $\Texp(\zeta,\xi)$ and $t(\zeta,\i)$.

%On the other side the approximation of $\psi|_{\partial\Omega}$ by an exponential is also good for small $|\zeta|$. {\bf Explain!}
% No, this is the other way around.  It is good for large $|\zeta|$.  
%Hence, $\Texp(\xi,\zeta)$ may differ significantly from $\t(\xi,\zeta)$ for large $|\zeta|.$ Second, $\Texp$ may not converge for $|\zeta| \rightarrow \infty,$ and even if it does the limit may not be in a reasonable function space (say $L^2(\R^3)$). Below we will  attempt to remedy these difficulties and build a practical numerical reconstruction algorithm using the above mentioned steps.

\subsection{Calderon's linearized reconstruction method} \label{sec:calderon}

Several properties of $\Texp$ can be
established from an analysis comparing this approach to that of
Calder\'on.  In \cite{knudsenLassasMuellerSiltanen07} a connection was established between the 2-D
D-bar method based on the global uniqueness proof by Nachman \cite{nachman96}
and Calder\'on's linearized reconstruction method. % We begin our
% calculations in an analogous manner and use them to show that
% $\Texp(\xi,\zeta)$ is actually independent of $\zeta$ and that
% $\Texp(0,\zeta) = 0$.

%From definition \eqref{Texp}, $\Texp$m is
%defined in terms of the Dirichlet to Neumann map for the conductivity
%equation by
%\begin{eqnarray}\label{Texp_again}
%%\Texp(\xi,\zeta)  =  \int_{\partial\Omega}e^{-ix\cdot(\xi+\zeta)}(\Lambda_\gamma-\Lambda_1) e^{ix\cdot \zeta}d\sigma(x).
%\Texp(\xi,\zeta)  =  \int_{\partial\Omega}e^{-ix\cdot(\xi+\zeta)}(\Lambda_\gamma-\Lambda_1) e^{ix\cdot\zeta}d\sigma(x).
%\end{eqnarray}
Define a function $\uexp(x,\zeta)$ as the unique solution to the boundary value problem
\begin{eqnarray*}
\nabla\cdot\gamma\nabla\uexp(x,\zeta) &= 0, \quad x\in \Omega, \quad \zeta\in \C^3 \\
\uexp\vert_{\DOm} &= e^{ix\cdot\zeta}.
\end{eqnarray*}
Integration by parts in equation \eqref{Texp} results in a formula for $\Texp$
defined in terms of $\gamma$ in the interior
\begin{eqnarray} \label{texp_int}
\Texp(\xi,\zeta) &=& \int_{\Omega}(\gamma - 1)\nabla\uexp(x,\zeta)\cdot\nabla e^{-ix\cdot(\xi+\zeta)}dx. %\\
%&=&\int_{\DOm}(\gamma-1)\frac{\partial \uexp}{\partial n}e^{-ix\cdot(\xi+\zeta)}ds \nonumber \\
%&-&\int_{\Om}\nabla\cdot((\gamma-1)\nabla\uexp(x,\zeta))e^{-ix\cdot(\xi+\zeta)} dx \nonumber
\end{eqnarray}
%and $\gamma=1$ on $\DOm$.
Write $\uexp = e^{ix\cdot\zeta} + \du$ for $\du\in H^1_0(\Om)$.  Then $\du$ satisfies
\begin{eqnarray}
\nabla\cdot(\gamma\nabla\du) &=& -\nabla\cdot((\gamma-1)\nabla e^{ix\cdot\zeta}) \label{du_pde},
\end{eqnarray}
and one can estimate
\begin{eqnarray}\label{du_estimate}
\|\du\|_{H^1(\Om)}\leq C\|(\gamma-1)\nabla
e^{ix\cdot\zeta}\|_{L^2(\Omega)} \leq |\zeta|\|\gamma - 1\|_{L^{\infty}(\Om)}e^{|\zeta|R},
\end{eqnarray}
where $R$ is the radius of the smallest ball containing $\Om$. From
\eqref{texp_int} we then get
\begin{eqnarray*}
  \Texp(\xi,\zeta) &= \int_{\Omega}(\gamma - 1)\nabla(e^{ix\cdot\zeta}
  + \delta u)\cdot\nabla e^{-ix\cdot(\xi+\zeta)}dx\\
  &= \int_{\Omega}(\gamma - 1)\nabla e^{ix\cdot\zeta}
  \cdot\nabla e^{-ix\cdot(\xi+\zeta)}dx + R(\xi,\zeta)\\
  &= (\xi\cdot \zeta) \int_{\Omega}(\gamma - 1)e^{-ix\cdot\xi}dx +
  R(\xi,\zeta),
\end{eqnarray*}
where the remainder term
\begin{eqnarray*}
   R(\xi,\zeta) = \int_{\Omega}(\gamma - 1)\nabla \delta u \cdot\nabla e^{-ix\cdot(\xi+\zeta)}dx.
\end{eqnarray*}
Since $(\xi + \zeta)^2 = \zeta^2 = 0$ we have $-\xi^2 = 2\xi \cdot
\zeta$ and hence
\begin{eqnarray}\label{FTgamma}
   \Texp(\xi,\zeta) =  -\frac{|\xi|^2}2 \widehat{(\gamma - 1)} (\xi) + R(\xi,\zeta).
\end{eqnarray}
The remainder is $\O (|\zeta|)$ for $\zeta$ small,  which can be seen from
\eqref{du_estimate}. This fact suggests that we use the minimal $\zeta
\in \V_\xi,$ that is
\begin{eqnarray*}
  \zeta_\xi = -\frac \xi 2 + i\zeta_I, \qquad \text{with }  \zeta_I
  \cdot \xi = 0,\; |\zeta_I| = \frac{|\xi|}2.
\end{eqnarray*}
Moreover, with this particular choice we can divide in \eqref{FTgamma}
by $|\xi|^2$  as the following proposition shows.
\begin{prop}
  Suppose $\gamma \in L^\infty(\Omega).$ Then
  $$|\Texp(\xi,\zeta_\xi)| = \O(|\xi|^2)$$ for small $|\xi|.$  
\end{prop}
\begin{proof}
  Note that $|\zeta_\xi|^2 =|\xi|^2/2.$ Since $\Lambda_{\gamma}$ maps constant functions to zero and has its range inside
  the space of mean free functions in $H^{-1/2}(\partial\Omega)$, we have that
  for small $|\xi|$
  \begin{eqnarray*}
    |\Texp(\xi,\zeta_\xi)| &=
    \left|\int_{\partial\Omega}(e^{-ix\cdot (\xi +
      \zeta_\xi)}-1)(\Lambda_\gamma-\Lambda_1)(e^{ix\cdot
      \zeta_\xi}-1)d\sigma(x)\right|\\
& \leq C |\xi + \zeta_\xi| |\zeta_\xi|,
  \end{eqnarray*}
and hence the particular form of $\zeta_\xi$ gives the conclusion.
\end{proof}
With the particular choice $\zeta =\zeta_\xi$ in \eqref{FTgamma} we
now neglect the term $R(\xi,\zeta_\xi)$ and divide by $-|\xi|^2,$ which
gives
\begin{eqnarray*}
   -2\frac{\Texp(\xi,\zeta_\xi)}{|\xi|^2} \approx  \widehat{(\gamma - 1)}(\xi).
\end{eqnarray*}
Introduce $\chi_B(\xi),$ the characteristic function on the ball
$|\xi|<B.$ With this function we remove high frequencies  and invert the Fourier transform.
This results in a linear reconstruction algorithm
\begin{eqnarray}\label{gammaapp}
  \gammaapp( x) = 1 -\frac {2}{(2\pi)^3} \int_{\R^3} \frac{\Texp(\xi,\zeta_\xi)}{|\xi|^2}
   e^{ix\cdot \xi} \chi_B(\xi)d\xi.
\end{eqnarray}
This formula is equivalent to the second inversion formula obtained by Calder\'on \cite[p.\ 72]{Calderon}. 

In summary the linear reconstruction algorithm consists of two steps:
\begin{enumerate}
\item Compute $\Texp(\xi,\zeta_\xi)$ by \eqref{Texp}.
\item Compute the reconstruction by \eqref{gammaapp}.
\end{enumerate}

This method is truly a linearization of the nonlinear reconstruction method
outlined in section \ref{sec:nonlinear}. As explained above $\Texp$
is a linearization of the first step on page
\pageref{nonlinearStep1}. Moreover, the computation of the quantity
\begin{eqnarray*}
  1 - \frac {1}{(2\pi)^3} \int_{\R^3}
  \frac{\Texp(\xi,\zeta_\xi)}{|\xi|^2} e^{ix\cdot \xi} \chi_B(\xi)d\xi
\end{eqnarray*}
linearizes the step $\hat q \mapsto \sqrt\gamma.$ Finally,
linearizing the square function gives \eqref{gammaapp}.   
\section{The case of a spherically symmetric conductivity} \label{sec:spherical}
%************************************************************

As a test problem it is of special interest to consider spherically
symmetric conductivities in the unit sphere.  In this case  the
scattering transform has certain symmetry properties described below. Moreover, the 
Dirichlet-to-Neumann map is described explicitly  in terms of
eigenvalues and eigenfunctions, which in this case are
the spherical harmonics. These properties will be
derived in this section.
% in $\Lambda_q$ and it is
% possible to compute Some of these simplifications are described
% in the following section. In particular it is possible in this case to
% compute eigenvaIn the article \cite{SMI00} the eigenvalues
% of the DtN map for two dimensional piecewise constant conductivities
% were expicitley formulated and then used to approximate the
% eigenvalues of smooth conductivities. Similarly, an explicit
% expression for the eigenvalues of the DtN of three dimensional domains
% can be found in the case of piecewise constant conductivities. We
% start by showing that the spherical harmonics are eigenfunctions and
% then proceed to the eigenvalues.  In the next chapter these
% eigenvalues are used to approximate the eigenvalues of the DtN map
% with smooth conductivity distributions.

\subsection{Symmetry in the scattering transform} \label{sec:symmetries}
The Fourier transform of a spherically symmetric function is
spherically symmetric itself. For the scattering transforms $\t$ and
$\Texp$ we have similar prperties. In the following we will tacitly
assume that $\zeta$ is either small or large such that $\txz$ is well-defined.
\begin{prop}
  Let $R\in SO(2)$ be arbitrary, and suppose $q(x) = q(Rx)$ for $x \in \Omega.$ Then
  \begin{eqnarray}\label{Tsymm1}
    \t(\xi,\zeta) = \t(R\xi,R\zeta), \quad \Texp(\xi,\zeta) = \Texp (R\xi,R\zeta)
  \end{eqnarray}
  In particular,
\begin{eqnarray}\label{Tsymm2}
  \t(\xi,\zeta_1) = \t(\xi,\zeta_2), \quad \Texp(\xi,\zeta_1) = \Texp(\xi,\zeta_2) 
\end{eqnarray}
for all $\zeta_1,\zeta_2 \in \V_\xi.$
\end{prop}
\begin{proof}
  We will prove the result for $\t$ only; for $\Texp$ the reasoning
  is similar. Let $\R \in SO(2).$ By the uniqueness of the CGO
  solutions, the rotational invariance of the Laplace operator, and the
  symmetry in $q$ we
  have $\psi(x,\zeta) = \psi (Rx, R\zeta).$ Consider the integral
  \eqref{t} and make the change of variables $R^Ty = x:$
  \begin{eqnarray*}
    \t(\xi,\zeta) &=   \int_{\Omega}
  e^{-ix\cdot(\xi+\zeta)}\psi(x,\zeta)q(x)dx \\ 
  &= \int_{R\Omega}
  e^{-i R^Ty \cdot (\xi+\zeta)}\psi(R^Ty,\zeta)q(R^Ty)d(R^Ty)\\
  &= \int_{\Omega}
  e^{-i y \cdot R(\xi+\zeta)}\psi(y,R\zeta)q(y)dy\\
  &= \t(R\xi,R\zeta).
  \end{eqnarray*}

To prove \eqref{Tsymm2} fix $\xi \in \R^3$ and take $\zeta_1,\zeta_2 \in \V_\xi$ with
$|\zeta_1| = |\zeta_2|.$ Then
\begin{eqnarray*}
  \zeta_j = \Re(\zeta_j) + i \Im(\zeta_j),\quad \Re(\zeta_j) & =-\frac
  \xi 2 + \xi_j^\perp \; \xi_j^\perp \cdot \xi = 0,\\
 \Im(\zeta_j) \cdot \xi &= \Im(\zeta_j) \cdot \xi^\perp = 0.
\end{eqnarray*}
Define a linear transformation $R$ by $R\xi = \xi,\; R \xi_1^\perp =
\xi_2^\perp,\; R (\xi \times \xi_1^\perp) = \xi \times \xi_2^\perp.$ As
a consequence $R \in SO(2)$ and \eqref{Tsymm2} follows from
\eqref{Tsymm1}.
\end{proof}

The Fourier transform of a real and even function is real itself. For
the scattering transform we have the following equivalent property:

\begin{prop}
  Suppose $q(x)$ is real and even. Then for $ \xi \in \R^3, \; \zeta \in \V_\xi$
  \begin{eqnarray}
    \ol{\t(\xi,\zeta)} = \t(\xi,\ol\zeta), \label{Tsymm3}\quad  \ol{\Texp(\xi,\zeta)} = \Texp(\xi,\ol\zeta).
  \end{eqnarray}
\end{prop}
\begin{proof}
We will again only show the properties for $\t.$  From the uniqueness
of the CGO solutions it follows that if $q$ is even then
$\mu(-x,\zeta) = \mu(x,-\zeta).$ Moreover, if $q$ is real, then
$\ol{\mu(x,\zeta)} = \mu(x,-\ol\zeta).$ Hence, if $q$ is both even and
real then
\begin{eqnarray*}
  \ol{\t(\xi,\zeta)} &= \int_\Omega e^{ix\cdot\xi} \ol {q(x)}
  \ol{\mu(x,\zeta)}dx  = \int_\Omega e^{-i(-x)\cdot\xi}  {q(x)}{\mu(x,-\ol\zeta)}dx\\
  & = \int_\Omega e^{-iy\cdot\xi}  {q(y)}
  \ol{\mu(y,\ol\zeta)}dy\\
  &=\t(\xi,\ol\zeta).
\end{eqnarray*}
\end{proof}
We now have a corollary for spherically symmetric
potentials:
\begin{cor}\label{Treal}
  Suppose $q$ is spherically symmetric. Then 
  \begin{eqnarray*}
    \ol{\t(\xi,\zeta)} =\t(\xi,\zeta),\quad \xi \in \R^3, \; \zeta \in \V_\xi.
  \end{eqnarray*}
\end{cor}
\begin{proof}
  There exists $R \in SO(2)$ such that 
  \begin{eqnarray*}
    R(\xi) =\xi,\quad R(\zeta) = \ol \zeta,
  \end{eqnarray*}
  and hence from \eqref{Tsymm1} we have $\t(\xi,\zeta)
  = \t(\xi,\ol\zeta).$ Equation \eqref{Tsymm3} now implies the result
  for $\t.$ For $\Texp$ the result follows similarly.
\end{proof}

\subsection{Eigenfunctions and eigenvalues for the Dirichlet-to-Neumann map}
%The following proof is included for the reader's convenience.
 
We use the same
ideas for the computation of eigenvalues for the 3-D problem that were used for the 2-D problem in  \cite{SMI00}. 

\begin{prop}
Let $D$ be the unit disk and suppose $\gamma(x)$ is  spherically symmetric. Then the eigenfunctions of $\Lambda_\gamma$
are the spherical harmonics $Y_l^m.$ 
\end{prop}
\begin{proof} 
  When $\gamma$ is spherically symmetric it follows from separation of
  variables that the solution to $\nabla \cdot
  \gamma \nabla u_{lm} = 0$ with $u_{lm}|_{\partial D} = Y_l^m$ is
  \begin{eqnarray}\label{eq:sphericalsol}
    u_{lm} = R_{l}(r)Y_l^m(\theta, \phi),
  \end{eqnarray}
  where $R_l(r)$ solves an Euler type equation.  Thus
\begin{eqnarray}
\Lambda_\gamma Y_l^m(\tp)= \Lambda_\gamma u|_{r = 1} =  \left .\gamma \frac{\partial R_l }{\partial r}\right|_{r = 1} Y_l^m(\tp) =  \lambda_l Y_l^m(\tp). \label{eval1}
\end{eqnarray}
\end{proof}
Note that $\lambda$ is independent of $m$ since $R_l$ is independent
of $m$.

%************************************************************
\subsection{Approximation of Eigenvalues and Eigenfunctions of the Dirichlet-to-Neumann Map}
\label{sec:eigvals}
%************************************************************
Next we will consider how to approximate the eigenvalues
for the special case of a constant conductivity $\gamma=1$. The particular form of $R_l$ gives
the following result.
\begin{prop} The eigenvalues of $\Lambda_1$ are given by $\lambda_l = l$.
\end{prop}

In the case of a piecewise constant radially symmetric
conductivity the eigenvalues can be computed recursively. Suppose  $0 = r_0 <r_1 <
r_2<\ldots r_{N-1} < r_N = 1$ and  for $j = 1,2,\ldots,N$
\begin{eqnarray}\label{gammapw}
  \gamma (x) = \gamma_j >0,\qquad |x| \in [r_{j-1},r_j].
\end{eqnarray}

\begin{prop}
  Suppose $\gamma$ is given by \eqref{gammapw}. Then the eigenvalues of $\Lambda_\gamma$ are given by
  \begin{eqnarray*}
    \lambda_0 = 0,\quad \lambda_l = l - \frac{2l+1}{1+ C_{N-1}}, \quad l >0
  \end{eqnarray*}
where $C_j =  w_j\frac{\beta_l\gamma_{j+1} \rho_j + \gamma_j}{\gamma_{j+1}\rho_j - \gamma_j}$ with $\rho_1 = 1$, $\rho_j = \frac{C_{j-1}+ w_j}{C_{j-1} - \beta_l w_j}$, $\beta_l = \frac{l+1}{l}$ and $ w_j = r_j^{-(2l+1)}$. 
\end{prop}
\begin{proof}  Since  $Y_0^0$ is a constant,  $\lambda_0 = 0$.
  The solution to $\nabla \cdot \gamma \nabla u_{lm} = 0, \; u_{lm}|_{\partial D} = Y_l^m,$ is given by \eqref{eq:sphericalsol} with $R_l(r) = A_j r^{l} + B_j
  r^{-(l+1)}$ for $r_{j-1} \leq r<r_j, \, j = 1, \ldots , N$. The
  coefficients $A_j$ and $B_j$ are determined by matching the
  Dirichlet and Neumann conditions at the $r_j$, $j=1,\ldots, N-1$. The outermost Dirichlet condition (at $r = 1$) gives
  $1 = A_N + B_N$ which leads to the following eigenvalue expression:
\begin{eqnarray}
\lambda_l = \gamma  \left .\frac {\partial v_{lm}}{\partial r}\right|_{\partial D} \quad
%&=&\gamma_N(l A_N r_N^{l-1}-(l+1) B_N r_N^{-(l+2)})  \nonumber \\
%\Lambda_\gamma Y_l^m = \Lambda_\gamma u_{lm} &=& \gamma_N  \left (l A_N r_N^{l-1}-(l+1) b_N r_N^{-(l+2)}\right) Y_l^m = \lambda_l Y_l^m \nonumber \\
%\gamma_N(l A_N r_N^{l-1}-(l+1) B_N r_N^{-(l+2)}) &=& \lambda_l 
%&&\mbox{ and with } r_N = 1 \mbox { and } \gamma_N =1 \nonumber \\
= lA_N -(l+1) B_N \quad
%&=&l(1-B_N) - (l+1) B_N \nn \\
= l -(2l+1)B_N \label{eigval} 
\end{eqnarray}

Moreover, by induction it follows that $A_j = B_jC_{j-1}$ for $j = 2,\ldots,N$. 
Again using the Dirichlet condition from the boundary, $1= A_N + B_N$ we get $B_N= (C_{N-1} +1)^{-1}$ which leads to the expression of the eigenvalue as stated in the theorem.
\end{proof}

By \cite{somersaloCheneyIsaacsonIsaacson91} if conductivities $\gamma_L$ and $\gamma_U$ are such that $\gamma_L(r) \leq \gamma_U (r)$
for all $r$, then the eigenvalues 
$\lambda_l^L$ and $\lambda_l^U$ of their corresponding Dirichlet-to-Neumann maps satisfy $\lambda_l^L \leq \lambda_l^U$. This gives a means for finding lower and upper bounds on the eigenvalues 
of a smooth function by finding the eigenvalues of piecewise constant function, $\gamma_L$ and $\gamma_U$ that satisfy $\gamma_L(r)
\leq \gamma (r) \leq \gamma_U(r)$.

\section{Implementation details}\label{sec:numerics}

\subsection{Numerical method for computing the scattering transform $\t(\xi,\zeta)$}
% Recall that for $\xi\in\R^3$, the scattering transform of the potential $q$ is defined by
% \begin{eqnarray} \label{tdef}
%   \t(\xi,\zeta) = \int_{\Omega}
%   e^{-ix\cdot(\xi+\zeta)}\psi(x,\zeta)q(x)dx.
% \end{eqnarray}
We compute the scattering transform $\t(\xi,\zeta)$ from the
definition \eqref{t} as a comparison to the $\Texp$ approximation and to study the reconstructions from an accurate scattering transform.  The computation requires that we solve the Lippmann-Schwinger equation \eqref{muLS} for $\mu(x,\zeta)$.    Hence we require
\begin{itemize}
\item A method of computation for the Faddeev Green's function in three dimensions
\item A numerical method for the solution of \eqref{muLS}
\item Numerical quadrature for computing $\t(\xi,\zeta)$ from \eqref{t}
\end{itemize}
 We describe each of these in turn.
 
\subsubsection{Computation of the Faddeev Green's function}

The Faddeev Green's function was defined in equations
\eqref{Gzeta}. The effect of scaling and rotation of $\zeta$ on
$G_\zeta$ was analyzed in \cite{corneanKnudsenSiltanen}, and it was
shown that when $\zeta$ satisfying $\zeta \cdot \zeta =0$ is
decomposed in the form
\begin{eqnarray}\label{zdecomp}
  \zeta = \kappa(k_\perp + i k),
\end{eqnarray}
where $k_\perp, k\in\R^3, |k_\perp| = |k|= 1, k \cdot k_\perp = 0,$
and $|\zeta| = \sqrt 2 \kappa,$
then
\begin{eqnarray}\label{sgzeta}
  g_\zeta(x) = \kappa^{n-2}g_{k_\perp + i k}(\kappa x).
\end{eqnarray}
Furthermore, if $R\in SO(2)$ then
\begin{eqnarray}\label{Rgzeta}
  g_\zeta(x) = g_{\RRR\zeta}(R x).
\end{eqnarray}
Combining \eqref{sgzeta} and \eqref{Rgzeta} yields the formula
\begin{eqnarray}
  g_\zeta(x) = \kappa g_{e_1+ie_2}(\kappa R x),
\end{eqnarray}
where $R\in SO(2)$ and the first and second column of $R$ is
$k,k_\perp$ respectively. This formula shows that it is sufficient to
compute $g_{e_1+ie_2}.$ 

To compute $g_{e_1+ie_2}$ we will use formula (6.4) of \cite{newton89}
\begin{eqnarray}\label{gzetaformula}
  g_{e_1+ie_2}(x) = \frac {e^{-r+x_2-ix_1}}{4\pi r} - \frac 1 {4\pi} \int_{s}^1 \frac {e^{-ru+x_2-ix_1}}{\sqrt{1-u^2}}J_1(r\sqrt{1-u^2})du,
\end{eqnarray}
where $J_1$ denotes the  Bessel function of the first find of order one.  Here $r = |x|$ and $s = \hat x \cdot e_2 =
x/|x|\cdot e_2.$ Since the function $J_1(t)/t$ is continuous on the
interval $[0,\infty)$ (in particular at $t = 0$), we will approximate the
integral in \eqref{gzetaformula} by a simple midpoint Riemann sum
\begin{eqnarray*}
  \int_{s}^1 \frac
  {e^{-ru+x_2-ix_1}}{\sqrt{1-u^2}}J_1(r\sqrt{1-u^2})du \approx
  \sum_{j=1}^N \frac
  {e^{-ru(j)+x_2-ix_1}}{\sqrt{1-u(j)^2}}J_1(r\sqrt{1-u(j)^2}) h,
\end{eqnarray*}
where $N$ is the number of discretization points, $h = (1-s)/N$ and
$u(j) = s + (j-1/2)h,\; j =1,2,\ldots, N.$

\subsubsection{The computation of complex geometrical optics}
Having computed the Faddeev Green's function we now turn to the
numerical solution of the integral equation \eqref{muLS} for
$\mu(\cdot, \zeta).$ We will use a method due to Vainikko  \cite{vainikko00} for solving
Lippmann-Schwinger equations; see also \cite{hohage,knudsenMuellerSiltanen04} for
implementations in different contexts.  The main idea is to transform \eqref{muLS} to a multiperiodic
integral equation in $\R^3,$ which can be solved efficiently using FFT. 

Let $G_\rho =
\{x \in \R^3\; | \;
|x_i| \leq \rho \}.$ Then by assumption $\supp(q) \subset \Omega \subset
G_1.$ Extend the potential $q$ and the Green's function $g_\zeta$
to $G_{2}$ such that
\begin{eqnarray*}
\qp(x) &= \left\{
\begin{array}{ll}
  q(x),  &x \in \Omega,\\
    0,  & x \in G_{2} \setminus \Omega,
\end{array}  \right.\qquad
  \gp_\zeta(x) &= \left\{
  \begin{array}{ll}
    g_\zeta(x),  &x \in \Omega,\\
    0,  & x \in G_{2} \setminus \Omega,
  \end{array}\right.
\end{eqnarray*}
and then extend  $\qp$ and $\gp_\zeta$ to $\R^3$ as
periodic functions in all variables with period equal to $4.$ Instead of \eqref{muLS} we consider the
periodic integral equation
\begin{eqnarray}\label{muLSperiodic}
  \mup(x,\zeta) + \int_{\R^3} \gp_\zeta(x-y)  \qp(y) \mup(y,\zeta) dy = 1.
\end{eqnarray}
This equation  is uniquely solvable  since  \eqref{muLS}  is, and
moreover one can show that on $\Omega$ we have
\begin{eqnarray*}
  \mup(x,\zeta) = \mu(x,\zeta),\qquad x\in\Omega.
\end{eqnarray*}
In order to solve \eqref{muLSperiodic} numerically define
\begin{eqnarray*}
  \Z_N^3 = \{ j \in \Z^3 \; | \; -N/2 \leq j_k < N/2,\; k = 1,\ldots,3 \}
\end{eqnarray*}
and the computational grid
\begin{eqnarray*}
  C_N = h \Z_N^3
\end{eqnarray*}
where $h =  4/N$ specifies the discretization fineness. Define the
grid approximation $\phi_N$ of a continuous function $\phi \in C(G_{2})$ by
\begin{eqnarray*}
  \phi_N(jh) = \phi(jh)
\end{eqnarray*}
and the grid approximation $g_N$ of $\gp_\zeta$ (which is smooth except for
a singularity at the origin) for fixed $\zeta$ by
\begin{eqnarray*}
  g_N(jh) = \left\{
  \begin{array}{ll}
    0, & j = 0\\
    \gp_\zeta(jh), & \text{otherwise}. 
  \end{array}\right.
\end{eqnarray*}
The convolution operator appearing in \eqref{muLSperiodic} 
\begin{eqnarray*}
  K \phi (x) = \int_{\R^3} \gp_\zeta(x-y)  \phi(y) dy
\end{eqnarray*}
is now discretized by trigonometric collocation, which, using the
discrete Fourier transform $\F_N$, gives 
\begin{eqnarray*}
   K_N \phi_N (jh) = \F_N^{-1} (\hat \gp_N \cdot \hat  \phi_N).
\end{eqnarray*}
Here $\cdot$ denotes pointwise multiplication. In practice the
discrete Fourier transform can be implemented efficiently using FFT (with
proper zero-padding) in $\O(N^3\log(N))$ arithmetic operations. The total
discretization of \eqref{muLSperiodic} now reads
\begin{eqnarray*}
    \mu_N + K_N(q_N\mu_N) = Q_N 1. 
\end{eqnarray*}
This discrete linear system is solved numerically in matlab using the
iterative algorithm GMRES \cite{saadSchultz}, without setting up a
matrix for the linear map $K_N(q_N\cdot).$

\subsubsection{The scattering transform}
Having computed the grid approximation $\mu_N$ it is straightforward
to evalute $t(\xi,\zeta)$ by using numerical integration in \eqref{t}. In this
implementation we have used a simple midpoint quadrature rule.

\subsection{Numerical method for computing  $\Texp$ for spherically symmetric conductivities}

For the calculation of $\Texp$ we expand $e^{ix\cdot \zeta}$ in terms
of spherical harmonics\footnote{We use here the normalized spherical
  harmonics given by $Y_l^m(\tp) = N_l^mP_l^m(\cos \theta) e^{im\phi}$
  where $N_l^m$ are normalization factors and $P_l^m$ are associated
  Legendre functions.} and $e^{-ix \cdot(\xi+\zeta)}$ in terms of the
spherical harmonics conjugates, 
\begin{eqnarray*} 
 e^{-ix \cdot(\xi+\zeta)}&= \sum_{l=0}^\infty \sum_{m=-l}^{l} a^*_{lm}(\xi,\zeta) [Y_l^m(\tp)]^*\\
e^{ix\cdot \zeta} &= \sum_{k=0}^\infty \sum_{n=-k}^{k} b_{kn}(\zeta) Y_k^n(\tp).
\end{eqnarray*}
Using these expansions leads to 
\begin{eqnarray}\label{eq:TexpSum}
\Texp(\xi, \zeta) =\sum_{l,m} \sum_{k,n}  a^*_{lm}(\xi,\zeta) b_{kn}(\zeta)  \int_{\partial D} [Y_l^m(\tp)]^* (\Lambda_q -\Lambda_0)Y_k^n(\tp) d\sigma
\end{eqnarray}

In the special case of spherically symetric conductivities we can use the knowledge of the eigenvalues of the Dirichlet-to-Neumann maps, $\Lambda_\gamma Y_l^m(\theta,\phi) = \lambda_l Y_l^m$ to simplify the calculation of $\Texp$. In particular we get 
\begin{eqnarray}
\Texp(\xi, \zeta) 
&=&\sum_{l,m,k,n}  a^*_{lm}(\xi,\zeta)b_{kn}(\zeta)   \int_{\partial D} [Y_l^m(\tp)]^* (\lambda_k-k)Y_k^n(\tp) d\sigma\nn \\
&=&\sum_{l,m,k,n}  a^*_{lm}(\xi,\zeta)b_{kn}(\zeta)  (\lambda_k-k) \int_{\partial D} [Y_l^m(\tp)]^* Y_k^n(\tp) d\sigma\nn \\
&=&\sum_{l,m} a^*_{lm}(\xi,\zeta) b_{lm}(\zeta) (\lambda_l-l) 
\label{texpassum}
\end{eqnarray}
The last equality comes from the orthonormality of the spherical harmonics. Equation (\ref{texpassum}) can be easily calculated if the coefficients $a^*_{lm}$ and $b_{lm}$  are available. 
In this work, these coefficients were calculated with a  software package called `\textit{S2kit}' which are C routines that can be accessed from Matlab. Detailed information can be found in \cite{fastHarmExp03}.

\subsection{Computation of the conductivity}
After taking the high frequency limit in \eqref{eq:zetalim} and \eqref{texpzetalim}  we calculate the inverse
Fourier transform to get $q(x)$ and $\qexp(x).$ The integral in the inverse Fourier
transform is here computed numerically using a simple Riemann sum. To get the conductivity $\gamma$ we
need to solve the boundary value problem $\Delta \gamma^{1/2} =
q \gamma^{1/2} $ with $\gamma^{1/2}|_{\partial \Omega} = 1$. This was
realized with the standard Green's function for the Laplace equation
in three dimensions. Using symmetries reduces the problem to a single
integral.

\subsection{Numerical implementation of Calder\'on's method}
 Calder\'on's method based on \eqref{gammaapp} is
simply implemented by evaluating the integral using numerical quadrature. 

%%********************************
\section{Results} \label{sec:examples}
%********************************
%%********************************
\subsection{Examples}
\label{sec:theexample}
%********************************
The conductivity distributions we will use in the examples are smooth,
spherically symmetric and constant one near $\partial\Omega$. They are given
by
 \begin{eqnarray}
 \gamma(x) &= (\alpha \Psi(|x|) +1)^2 \nonumber \\
 \Psi(r) &= \left \{ \begin{array}{cl} e^{-\frac{r^2}{(r^2-d^2)^2}}
     \quad &\mbox{for } -d<r<d \\
 0 \quad & \mbox{otherwise} \end{array} \right.\label{Psi}
 \end{eqnarray}
 where $0<d\leq 1$ is a parameter determining the support of $\Psi.$
 The parameter $\alpha$ regulates the amplitude of $\gamma$, which is
 largest at $r= 0$ with amplitude $(\alpha +1)^2$.  A similar function
 was used \cite{SMI00} as an example for the two dimensional problem.
% Figure~\ref{plottruegam} shows the conductivity distributions
% corresponding to different values of $d$ and $\alpha$.
%\begin{figure}[h!]
%\begin{center}
%\epsfxsize=6cm
%{\epsffile{gammaDiffDAlp5col.eps}}
%\epsfxsize=6cm
%{\epsffile{gammaDiffAlphaD09col.eps}}
%\epsfxsize=12cm
%{\epsffile{exampleDifDDifAlp.eps}}
%\caption[Plots of different conductivity distributions]{Left: $\gamma(r)$ as the parameter $d$  varies ($\alpha =
%  0.5$).  Right: $\gamma(r)$ as $\alpha$ varies
%  ($d = 0.7$).}
%\label{plottruegam}
%\end{center}
%\end{figure}

 \subsection{The scattering transform} Let us fix $d = 0.9,\alpha =
 0.3.$ in \eqref{Psi}.  We are interested in the limit of $\txz, \xi
 \in \R^3,\zeta \in \V_\xi,$ when $|\zeta|$ goes to infinity. For
 purpose of illustration we compute $\txz$ for fixed $\xi =
 (10,0,0)$ and varying $\zeta \in\V_\xi$ with $8<|\zeta| <50.$ We use
 a discretization level in the algorithm corresponding to $N=2^6.$ In
 addition we compute $\Texp(\xi,\zeta)$ by \eqref{eq:TexpSum} using the first 30 eigenvalues
 of the Dirichlet-to-Neumann map.  We truncate the sum of the
 spherical harmonics at $l=30,$ which means we use approximately the
 first 900 spherical harmonics. As a benchmark we compute
 $\qhat(\xi).$ The results are shown in figure
 \ref{diffMagOfZeta}.
\begin{figure}[h!]
\begin{center}
\epsffile{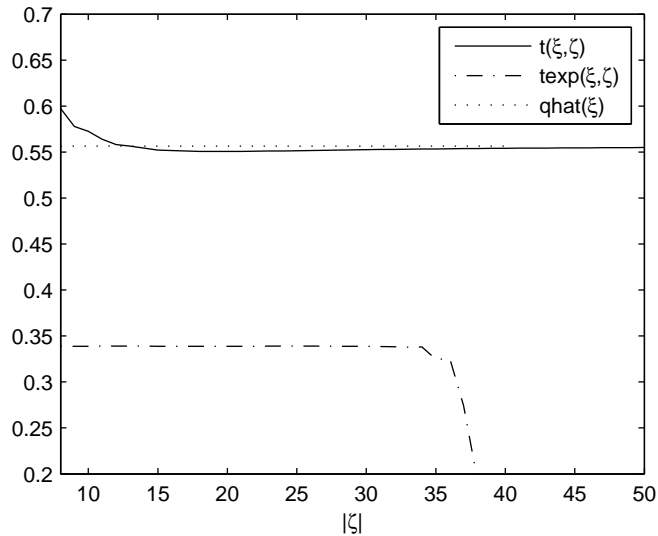}
\caption[Scattering data $\txz$ with varying $|\zeta|$]{$\t(\xi,\zeta),$
  $\Texp(\xi,\zeta)$ calculated for fixed $\xi = (10, 0,0)$ and
  varying $|\zeta|$. Here $d=0.9$ and $\alpha=0.3$.}
\label{diffMagOfZeta}
\end{center}
\end{figure}
We know from Corollary~\ref{Treal} that $\t$ and $\Texp$ are real and
this is consistent with our numerical results.  The data verifies that
for our example $\txz$ converges to $\hat q (\xi)$ as $\zeta
\rightarrow \infty.$ We observe that $\Texp$ is independent of the
magnitude of $\zeta \in \V_\xi,$ until it diverges due to numerical instability. The same phenomena appears in other examples and with
different values of $\xi.$ We believe that this phenomena has to do
with the special class of spherically symmetric conductivities considered
here.

Next we compare $\txz$ and $\Texp(\xi,\zeta)$ for different values of
$\xi$.  For each $\xi = s[1,0,0],\; s \in [0, 50],$ we fix $\zeta\in
\V_\xi$ with $|\zeta| = 50.$ We compute $\txz$ using a discretization
level with $N = 2^6.$ $\Texp$ is computed with the parameters as
above.   As a benchmark we compute $\hat q(\xi).$ The results are displayed in
figure \ref{qhat}.  
\begin{figure}[h!]
\begin{center}
{\epsffile{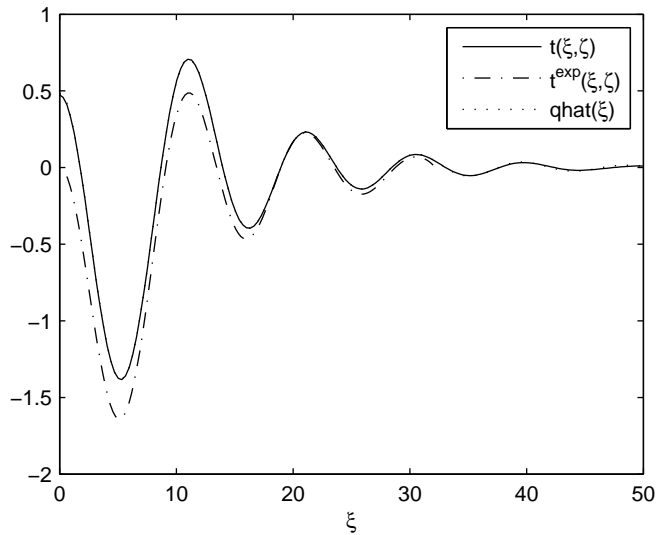}}
\caption[Scattering data]{Scattering data $\t,$ $\Texp$ and $\hat q$
  (qhat) with $d=0.9$ and $\alpha=0.3$. For each $\xi,$ $\zeta\in \V_\xi$ is
  chosen such $|\zeta| = 50.$ The Fourier transform $\hat q$ virtually coincides with
  $\t(\xi,\zeta)$. }
\label{qhat}
\end{center}
\end{figure}
The difference in $\hat q(\xi)$ and $\t(\xi,\zeta)$ is very small.
$\Texp(\xi)$ is displayed only for $0\leq |\xi|\leq 32$ since the
calculation becomes numerically unstable and blows up for $|\xi|> 32$.
One observes good agreement of all three curves for $|\xi|\geq
20$. Close to $|\xi| = 0$ the approximation $\Texp$ is close to zero
and differs from the correct values.

%********************************
\subsection {The reconstructions}
%********************************

Evaluting the inverse Fourier transform of the numerically computed $\Texp(\xi)$ and
$\t(\xi,\zeta)$ gives two approximations of $q(x)$ which are displayed in figure \ref{q}.  The
approximation calculated from $\txz$
differs as expected only slightly from the actual value. The
approximation $\qexp$ of $q$ calculated from $\Texp$ (and hence from
the boundary data) is quite different from $q.$ For $x$ near the boundary the $\qexp(x)$ is quite accurate,
but for $x$ near zero there are large discrepancies, especially in the
magnitude.  Looking at the scattering data in figure \ref{qhat}, one
sees two features most likely responsible for that difference.  The
first one is the differences in the values of $\Texp(\xi)$ for $\xi$
close to zero compared to $\hat q(\xi)$.  The second is the truncation
of $\Texp(\xi)$ due to numerical instability for large $\xi$ values.
More details on the influence of the truncation of the scattering data
are provided in section \ref{sec:trunc}.
\begin{figure}[h!]
\begin{center}
{\epsffile{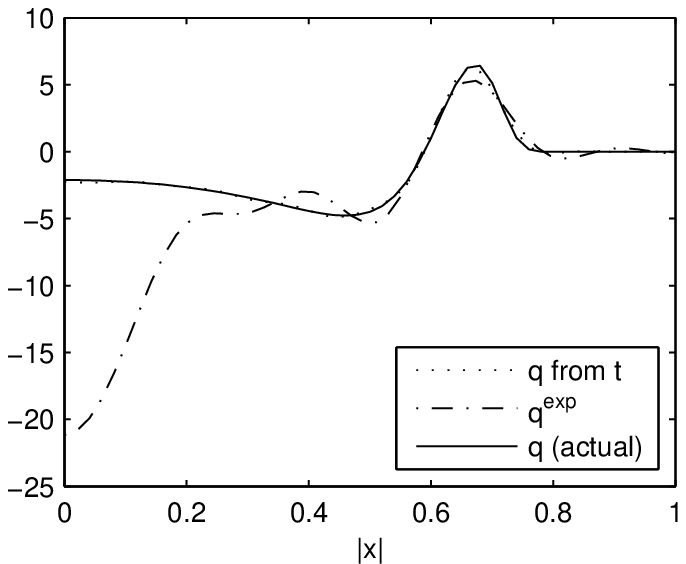}} \hfill {\epsffile{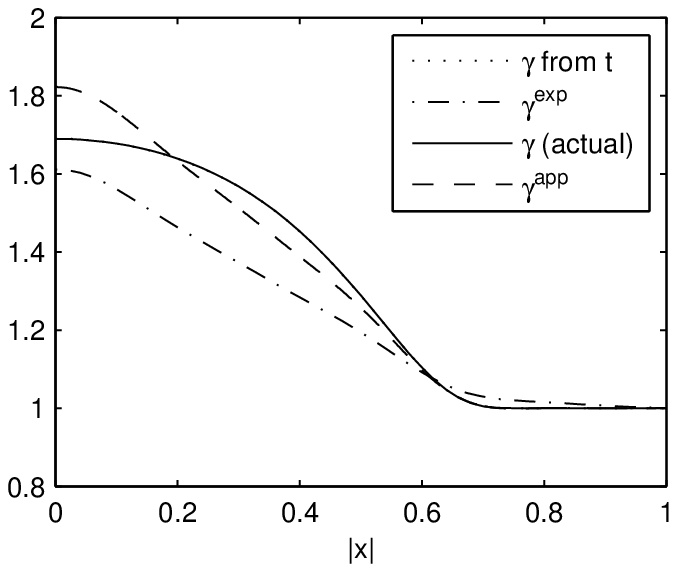}}
\caption[Reconstructions of the Schr\"odinger potential]{Left:  Reconstructions of $q(x)$ by taking the inverse Fourier transform of $\t(\xi,\zeta)$ and  $\Texp(\xi,\zeta)$  for $\alpha = 0.3$ and $d =0.9.$  Right:   Reconstructions of $\gamma$ from $\t,$ $\gammaexp,$ and $\gammaapp$ compared to actual
  conductivity  for $\alpha = 0.3$ and $d = 0.9.$ $\gamma$ from $\t$ nearly coincides  with the actual conductivity.}
\label{q}
\end{center}
\end{figure}

Also in figure \ref{q} we display three reconstructions of the
conductivity distribution. The first reconstruction of $\gamma$ is from
$\txz$. Since $\txz$ is computed from the forward problem, it may be expected that this reconstruction would be very close to the actual value, as it is. The second
reconstruction is $\gammaexp(x)$ from $\Texp,$ and the third
reconstruction $\gammaapp$ is from the linear method
\eqref{gammaapp}. Considering the relatively large difference in
magnitude of $\qexp(x)$, the reconstruction $\gammaexp$ is
surprisingly good. Also $\gammaapp$ is a fairly good reconstruction. A
positive aspect in both reconstructions is that we get $\gamma \equiv 1$ close to the
boundary. Moreover, the overall shape is also fairly well
reconstructed.
%\begin{figure}[h!]
%\begin{center}
%\epsfxsize=6cm
%{\epsffile{qD9A3hcol.eps}}
%{\epsffile{qfromt.eps}} \hfil 

% {\epsffile{gammafromt.eps}}
%\caption[Reconstructions of the Schr\"odinger potential]{
%  Reconstructions of $\gamma$ from $\t,$ $\gammaexp,$ and $\gammaapp$ compared to actual
%  conductivity  for $\alpha = 0.3$ and $d = 0.9.$ $\gamma$ from $\t$
%  coincides  with the actual conductivity.}
%\label{gamma}
%\end{center}
%\end{figure}

%We still need to keep in mind that $\gammaexp$ is the only reconstruction that does not use the knowledge of $\gamma$ explicitly. Implicitly it is in the eigenvalues which serve as our DtN map.
%\begin{figure}[htb]
%\begin{center}
%\epsfxsize=12cm
%{\epsffile{gamD9A3col.eps}}
%{\epsffile{gamD9A3bw.eps}}
%\caption[Reconstructions of square root of conductivity  distribution]{The square root of the reconstructed conductivity distribution with parameter $d = 0.9$ and $\alpha = 0.3$. The reconstructions from $\t$ and $\hat q$ are identical with $\sqrt{\gamma(x)}$.  }
%\label{gam}
%\end{center}
%\end{figure}

%*****************************
\subsection {The influence of the truncation of the scattering data}
\label{sec:trunc}
%*****************************

When we reconstructed $\qexp$ and $\gammaexp$ we truncated the scattering data $\Texp$ due to numerical instabilities. In this section we investigate the 
influence on the reconstructions of the truncation of the true scattering data $\txz.$
\begin{figure}[h!]
 {\epsffile{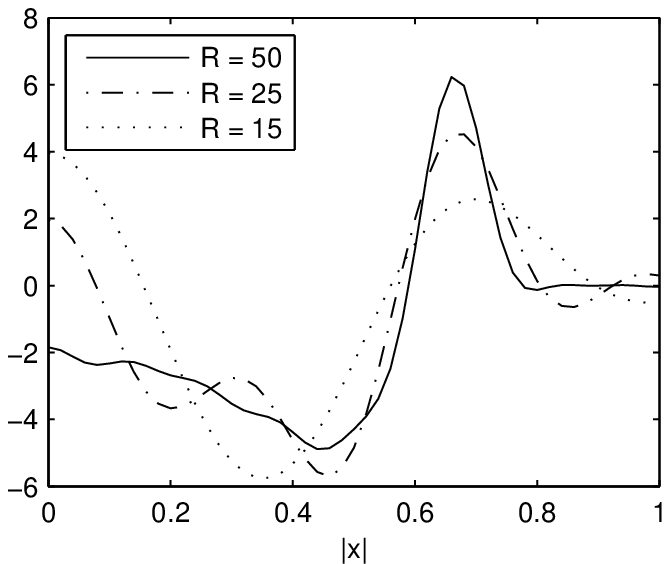}} \hfill {\epsffile{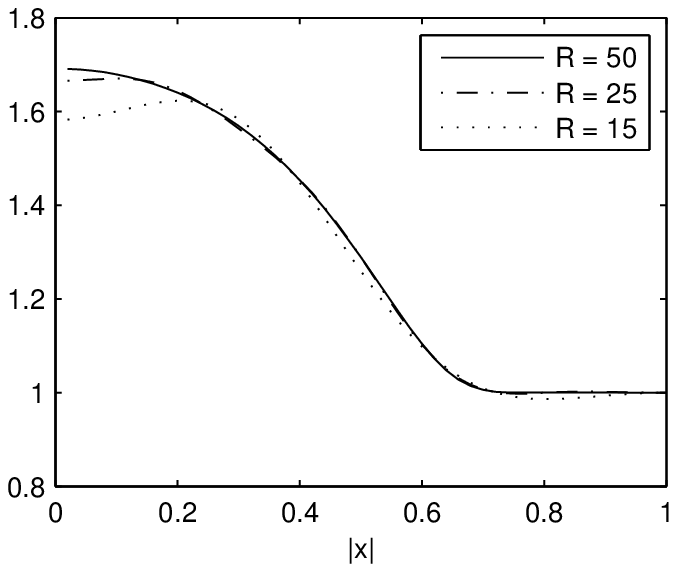}}
\caption[]{Left: reconstructed
   Schr\"odinger potential with truncation of $\t$ at $R=15, 25,$
   and $50.$ Right: reconstructions of ${\gamma}$.}
\label{fig:difftrunc}
\end{figure}
%{\bf I think we need to include the actual $q$ in the middle graph in Fig 5 and the actual $\gamma$ in the rightmost graph in Fig. 5}
Figure \ref{fig:difftrunc} shows $\t(\xi,\zeta)$ and the reconstructions
$q(x)$ and ${\gamma(x)}$ for different truncations of $\t(\xi,\zeta)$,
namely at $\xi = R$ for $R = 15, 25, 50.$ We have chosen $\zeta \in
\V_\xi$ with $|\zeta| = 50.$ The actual potential and conductivity are
almost identical to the curves corresponding to $R=50.$  It is evident that the amount of truncation
of the scattering transform  influences the reconstruction, and
that a very poorly reconstructed $q$ can still result in a good
approximation of $\gamma$.
 % First observe that $\t$ is not perfectly reconstructed for large $\xi$'s which is due to a coarse mesh of $x$ in the calculation of $\psi(x,\zeta)
%  $\footnote{The coarse mesh was used to save computation time.} We could calculate $\txz$ more accurately but in this context it shows an interesting phenomena. The inaccuracy of $\txz$ is 
%  reflected in oscillations in the Schr\"odinger potential (center plot of figure \ref{fig:difftrunc}) but are smoothed out in the reconstruction of $\gamma$ (third plot in figure \ref{fig:difftrunc}). 
%The computed conductivity is shown in the third plot of figure \ref{fig:difftrunc}. All reconstructions are very good. The  reconstruction of $\gamma$ from $\t$ with no truncation shows nearly no difference from the actual value of $\gamma$ even though the corresponding Schr\"odinger potentials differ.
%The only reconstruction of $\gamma$ showing a noticeable difference to the actual value is the one for which we truncate $\t$ most.
This suggests that for the reconstruction of $\gamma$ the values of the scattering data for small $\xi$ are very important.  
This is analogous to observations made in the 2-D case \cite{MS03}.

%*****************************
\subsection {Influence of the support and magnitude of $\gamma(x)^{1/2}-1$}
%*****************************
So far we have used fixed values for $d$ and $\alpha$, which determine
the support and the magnitude of $\gamma(x)^{1/2}-1$.  Figure
\ref{fig:gamexpdiffddiffa} displays the reconstructions $\gammaexp$
and $\gammaapp$ of
$\gamma(x)$ from $\Texp(\xi)$ for different choices of support $d$ and
magnitude $\alpha.$  Each row corresponds to a certain $d$-value
and each column to a specific $\alpha$-value.  For small support and
small magnitude we get good reconstructions, but the quality changes dramatically with
larger amplitude and larger support. Especially $\gammaexp$ does not
recover the actual conductivity very well for the large amplitude
$\alpha = 0.9.$ 
\begin{figure}[h!]
\begin{center}
{\epsffile{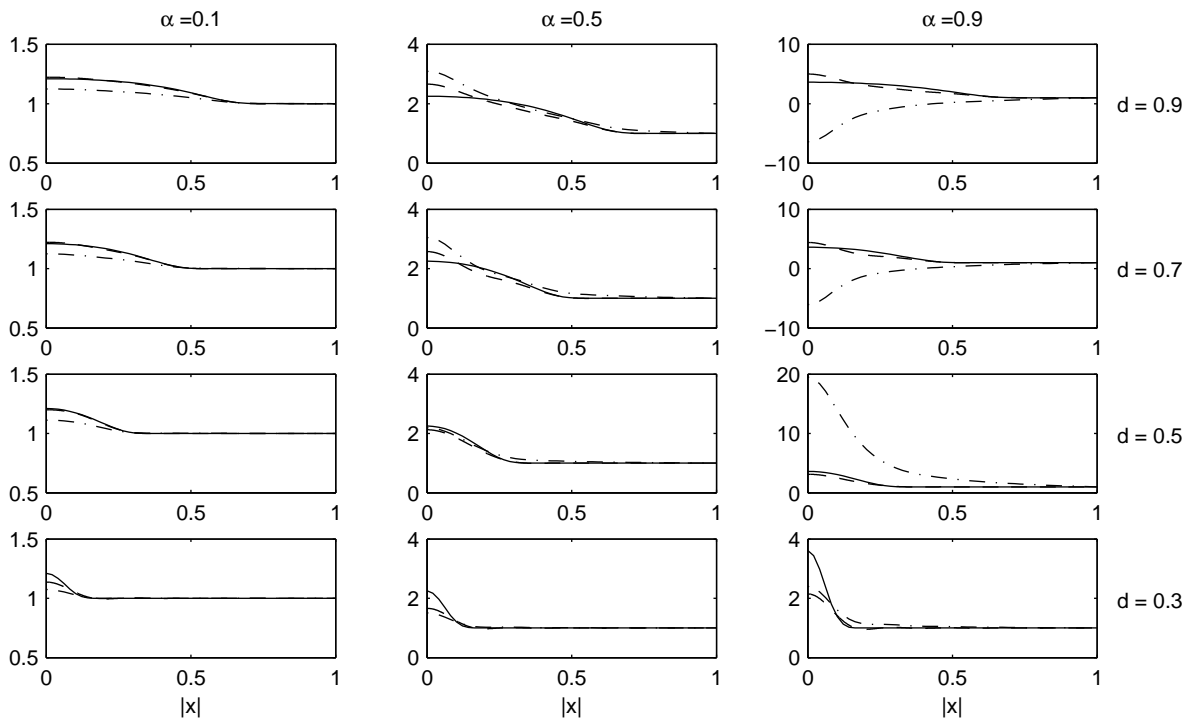}}
\caption[Reconstructions with different support and magnitude of
$\gamma$]{Reconstructions of conductivities of varying support and magnitude: each row corresponds to a specific support $d$
and each column corresponds to a specific magnitude of $\gamma$. The dash-dotted
curves are the $\gammaexp$ reconstructions, the dashed curves are the
$\gammaapp$ reconstructions, and solid curves are the
actual conductivities $\gamma.$}
\label{fig:gamexpdiffddiffa}
\end{center}
\end{figure}

\section{Conclusions}
In this work a direct method based on \cite{nachman88} for
reconstructing a 3-D conductivity distribution from the
Dirichlet-to-Neumann map was implemented and tested on noise-free data. A linearizing approximation
to the scattering transform, denoted $\Texp$ was studied and compared
to Calder\'on's reconstruction algorithm.  Reconstructions of
spherically symmetric conductivities in the unit sphere were computed
using the $\Texp$ approximation, Calder\'on's method, and a scattering
transform computed from the definition requiring knowledge of the
actual Schr\"odinger potential.  The latter case served as a benchmark
to study the quality of reconstructions for which the actual
scattering transform is known.  It was shown that very accurate
reconstructions can be obtained from accurate knowledge of the
scattering transform.  It was found that in contrast to the 2-D case,
the $\Texp$ approximation is inaccurate near the origin, and this
results in poor approximations to the magnitude of the conductivity.
However, the support of $\gamma-1$ and the boundary value $\gamma=1$
was well approximated by all three methods.  Truncating the computed
scattering transform in the computations was found to have a profound
effect on the reconstructed Schr\"odinger potential $q$, but the
affect on the reconstructed conductivity $\gamma$ was less dramatic.
In summary, it appears that the use of the full scattering transform
in this method is a promising approach for 3-D reconstructions, while
linearizations lead to significant inaccuracies in the reconstructed
amplitudes.

\section*{Acknowledgments}
The authors thank D. Isaacson and G. Boverman for helpful discussions on the spherical harmonics. 
This material is based upon work supported by the National Science Foundation under Grant No. 0513509 (J. Mueller).

\section*{References}
%\bibliographystyle{amsalpha} 

%\bibliography{litteratur}

\providecommand{\bysame}{\leavevmode\hbox to3em{\hrulefill}\thinspace}
\providecommand{\MR}{\relax\ifhmode\unskip\space\fi MR }
\providecommand{\MRhref}[2]{%
  \href{http://www.ams.org/mathscinet-getitem?mr=#1}{#2}
}
\providecommand{\href}[2]{#2}

\end{document}